\documentclass{IEEEtran}
\usepackage[dvips]{graphicx}

\usepackage{amssymb,amsmath,array,amsthm}
\usepackage{graphicx}
\usepackage{epstopdf}
\newtheorem{theorem}{Theorem}[section]
\newtheorem{lemma}[theorem]{Lemma}

\newtheorem{corollary}[theorem]{Corollary}
\usepackage{color}

\newtheorem{remark}[theorem]{Remark}

\begin{document}
\title{Construction of Structured Incoherent Unit Norm Tight Frames}
\author{{Pradip Sasmal, Phanindra Jampana and C. S. Sastry}%
\thanks{P. Sasmal and C. S. Sastry are with the Department of Mathematics,
  Indian Institute of Technology, Hyderabad, Telangana. 502285, India, email:
  ma12p1005@iith.ac.in, csastry@iith.ac.in}%
\thanks{P. Jampana is with the Department of Chemical Engineering,
  Indian Institute of Technology, Hyderabad, Telangana. 502285, India, email:
  pjampana@iith.ac.in}}
\maketitle
\begin{abstract} 
  The exact recovery property of Basis pursuit (BP) and Orthogonal Matching
  Pursuit (OMP) has a relation with the coherence of the underlying frame. A
  frame with low coherence provides better
  guarantees for exact recovery. In particular, Incoherent Unit Norm Tight
  Frames (IUNTFs) play a significant
  role in sparse representations. 
  IUNTFs with special structure, in particular those given by a
  union of several orthonormal bases, are known to satisfy better
  theoretical guarantees for recovering sparse signals. In the present work,
  we propose to construct structured IUNTFs consisting of large number of orthonormal bases. For a given $r, k, m$ with $k$
  being less than or equal to the smallest prime power factor of $m$
  and $r<k,$ we construct a CS matrix of size $mk \times (mk\times m^{r})$
  with coherence at most $\frac{r}{k},$ which consists of $m^{r}$ number of
  orthonormal bases and with density $\frac{1}{m}$. 
	We also present numerical results of recovery performance of union of orthonormal bases as against their Gaussian counterparts. 

\end{abstract}
\section{Introduction}
Frames are overcomplete spanning systems which are a generalization of
bases~\cite{chris_2003, Gita_2013}. 
A family of vectors $\{\phi_{i}\}^{M}_{i=1}$ in $\mathbb{C}^{m}$ is called a
frame for $\mathbb{R}^{m}$,
if there exist constants $0 < A \leq B < \infty$ such that
\begin{align*}
  A\left\|z\right\|^{2} \leq \sum^{M}_{i=1}\left|\left\langle z,
  \phi_{i}\right\rangle\right|^{2} \leq B\left\|z\right\|^{2}, \forall z \in \mathbb{C}^{m} 
\end{align*}
\noindent where $A,B$ are called the lower and upper frame bounds respectively \cite{chris_2003}.
By taking the frame vectors as columns, a full row rank matrix is obtained.
In the rest of the paper, we do not make any distinction between a frame
and its associated matrix and the use the two terms interchangeably. The characterization of a few frames is given in the
following. 
\begin{itemize}
\item If $A=B$, then $\{ \phi_{i} \}^{M}_{i=1}$ is called an $A-$tight frame or simply a
  tight frame.
\item If there exists a constant $c$ such that $\left\|\phi_{i}\right\|_{2} = c$
  for all $i= 1,2,\ldots, n,$ then $\{ \phi_{i} \}^{M}_{i=1}$ is an equal norm frame.
  If $c=1,$ then it is called a unit norm frame.
\item If a frame is both unit norm and tight, it is called a unit norm tight
  frame (UNTF). 

\end{itemize}

UNTFs are known to have good conditioning and provide stable representation. A UNTF exists only for $A = \frac{M}{m}$. It can be noted that a frame which is a concatenation of orthonormal bases is also a UNTF. The coherence of a frame is defined as the maximum absolute value of inner-product between two distinct normalized frame vectors. A UNTF with small
coherence is termed as an incoherent unit norm tight frame (IUNTF).

Compressed Sensing (CS)
\cite{can_2008,elad_2010} is a relatively
new paradigm in signal processing, which aims at recovering sparse signals
from very few linear measurements. Orthogonal Matching Pursuit
and Basis Pursuit (BP) are two of the most widely used CS algorithms. The
performance of both these algorithms depends on the coherence of the underlying
frame.

In \cite{elad_2002,feuer_2002,tropp_2004, grib_2002}, it is shown that frames
which are
a concatenation of several orthonormal bases provide better theoretical
recovery guarantees when compared to general frames. 
In some applications of image/audio processing \cite{molla_2004, starck_2004},
modeling of data as the superposition of several layers attains importance,
which implies the significance of an over-complete representation in terms of
union of orthonormal bases. Further, the special structure of underlying frames
allows for generating sparse representations through efficient solvers such as
block coordinate relaxation (BCR) \cite{sardy_2000}.

 However, it is very
difficult to construct a frame with small coherence which consists of large
number of orthonormal bases in $\mathbb{C}^m$.
Most of the existing constructions are dictated by some particular family
of numbers
(especially primes or their powers). In \cite{cal_1997, stro_2003}, the
authors have
constructed $m+1$ number of orthonormal bases for $\mathbb{R}^{m}$ with
coherence $1/\sqrt{m},$ where $m$ is a power of two. Some of
the well known structured IUNTFs are mutually unbiased bases (MUBs)
(\cite{klap_2003,san_2004,sey_2014,sey_2011,woc_2004}). Two orthonormal bases
$B$
and $B'$ of an $m-$dimensional complex inner-product space are called mutually
unbiased if and only
if $|\left\langle b, b' \right\rangle|^{2}=\frac{1}{m}$ for all $ b \in B$
and $b' \in B'.$
At most $m+1$ mutually unbiased bases of $\mathbb{C}^{m}$ can exist. If $m$ is
a power of a prime,
extremal sets containing $m+1$ mutually unbiased bases are known to exist
\cite{klap_2003,san_2004}.
However, to the
best of our knowledge there exist no constructions of union of
orthonormal bases with small coherence for more general sizes.

In this paper, we provide constructions for structured IUNTFs, more
specifically, concatenation
of orthonormal bases with small coherence, first for sizes governed by primes
or their powers and
then for composite dimensions using polynomials over
finite fields and recently introduced composition rule for binary matrices~\cite{pra_2016}. 

The paper is
organized in several sections. In Section~\ref{sec:basic}, we briefly review
the basics of compressed sensing. Section~\ref{sec:concortho} 
lists the enhanced recovery properties for frames which are a union of
orthonormal bases. Section~\ref{sec:consprimepower} discusses our
construction for sizes governed by primes or their powers.
In Section~\ref{sec:conscomp}, we describe the construction methodology
for general sizes using a recently proposed composition rule for binary
matrices~\cite{pra_2016}.

\section{Basics of Compressed Sensing}
\label{sec:basic}
\subsection{Compressed Sensing}
Compressed Sensing (CS) aims to recover a sparse signal $x\in \mathbb{R}^{M} $
from a few of its
linear measurements $y\in \mathbb{R}^{m} $. A vector is called sparse if only
a few of its elements
are non-zero. Sparsity is measured using the $\| \cdot \|_{0}$ norm,
$\|x\|_{0}:=|\{j\in\{1,2,\dots,M\}:x_{j}\neq 0\}|$. A signal $x$ is said to be
$s$-sparse if $\|x\|_0 \leq s$. 
The measurement vector $y$
is obtained from
the linear system $y=\Phi x$, where $\Phi$ is an $m \times M \;(m < M)$ matrix.
Sparse
solutions can be obtained by the following minimization problem,
\begin{displaymath}
P_0(\Phi,y):\min_{x} \Vert{x} \Vert_0 \; \mbox{subject to} \quad \Phi x=y.
\end{displaymath} 
However, $P_0(\Phi,y)$ is combinatorial in nature and is known to be
NP-hard~\cite{bourgain_2011}.
A common way to obtain approximate solutions for $P_0$ is by using greedy
methods~\cite{tropp_2004}.
Another approach is to solve a convex relaxation of
$P_{0}(\Phi,y)$ (\cite{can_2008}),
\begin{displaymath}
P_1(\Phi,y):\min_{x} \Vert{x} \Vert_1 \; \mbox{subject to} \quad \Phi x=y.
\end{displaymath}
\noindent The coherence of the matrix $\Phi$ is defined as
$$\mu_\Phi= \max_{1\leq\; i,j \leq\; M,\; i\neq j}
\frac{|\; \phi_i ^T\phi_j|}{\Vert \phi_i\Vert_{2} \Vert
  \phi_j \Vert_2},$$  
which gives bounds on the guaranteed
recovery of sparse signals via
Orthogonal Matching Pursuit (OMP) and Basis Pursuit (BP)~\cite{tropp_2004}. 
   
\noindent
\begin{theorem}\cite{elad_2010}
  An arbitrary $s-$sparse signal $x$ can be uniquely recovered
  as a solution to problems $P_{0}(\Phi,y)$ (using OMP and BP) and
  $P_{1}(\Phi,y)$,
  provided 
  \begin{equation} \label{eq:OMP_bound}
s < \frac{1}{2}\biggl(1+\frac{1}{\mu_\Phi}\biggl).
\end{equation}
\end{theorem}

The density of the frame $\Phi$ is key to minimizing the computational
complexity associated with the matrix-vector multiplication. Here,
by density, one refers to the ratio of number of nonzero entries to the total
number of entries of the matrix. The frames constructed in this paper have
small density, which aids in faster processing.

\section{Recovery Guarantees for Concatenation of Orthonormal Bases}
\label{sec:concortho}
Union of orthonormal bases provides better recovery properties compared to general frames. The enhanced recovery
properties for both OMP and BP are given below.
\begin{theorem}\cite{grib_2002}
  Suppose a frame $\Phi$ is a union of $Q$ orthonormal bases such that its coherence is
  $\mu_{\Phi}$. Let $x$ be a superposition of $s_{i}$ atoms
  from the $i$-th basis, $i = 1, \dots ,Q.$ Without loss of generality,
  assume that $0 < s_{1} \leq  s_{2} \leq \dots \leq s_{Q}$. Then OMP and BP recover the signal $x$ provided
  \begin{equation*}
    \sum^{Q}_{i=2}\frac{\mu_{\Phi}s_{i}}{1+\mu_{\Phi}s_{i}} < \frac{1}{2(1+\mu_{\Phi}s_{1})}.
  \end{equation*}
   
\end{theorem}

\begin{corollary}\cite{elad_2002}
Suppose that $\Phi$ is a concatenation of two orthonormal bases with coherence $\mu_{\Phi}$, and let $x$ be a
signal consisting of $s_1$ atoms from the first basis and $s_2$ atoms from the second basis, where $s_1 \leq s_2.$
Then the above condition holds whenever
\begin{equation*}
 2\mu^{2}_{\Phi}s_1s_2 + \mu_{\Phi}s_2 < 1.
\end{equation*}
\end{corollary}

\begin{theorem}\cite{grib_2002}
If $\Phi$ consists of $Q$ orthonormal bases, then OMP and BP recover any $s-$sparse signal provided
\begin{equation}
\label{eq:3}
s<\bigg[\sqrt{2}-1+\frac{1}{2(Q-1)}\bigg]\mu^{-1}_{\Phi}.
\end{equation}
\end{theorem}
For small values of $Q$ the bound in~(\ref{eq:3}) is
less restrictive than the general bound given in (\ref{eq:OMP_bound}).

\section{Construction method for prime and prime power sizes}
\label{sec:consprimepower}
In this section, we provide our construction method for a union of orthonormal
bases for the case
when $m$ is a prime or a prime power. Consider the finite field
$\mathbb{F}_{p}=\{f_{1},f_{2},\dots,f_{p}\}$
where $p$ is a prime or a prime power.
Let $S^{p}$ be the collection of polynomials of degree at most
$r$ (where $r < p-1$), which do not contain the constant term. It is easy to
check that the cardinality
of $S^{p}$ is $|S^{p}|=p^{r}.$  For $P\in S^{p},$ define the set
$S^{p}_{P}=\{P_{j}=P+f_{j}:j=1, \dots, p\}$. Fix any ordered $k-$tuple
$z\in \mathbb{F}_p^k$ with $r<k\leq p$. For simplicity, we consider
$z=(f_{1}, \dots, f_{k})$.
An ordered $k-$tuple is formed
after evaluating $P_{j}$ at each of the points of $z$ i.e,
$d^{P}_{j} := \big(P_{j}(f_{1}), \cdots, P_{j}(f_{k})\big)$.
From the $k-$tuple
$d^{P}_{j}$ we form a binary
vector $v_j^{P}$ of length $pk$ using
$$
v_j^{P}(p(m-1)+n) =
\begin{cases}
  1, \text{~if~} P_j(f_m) = f_n\\
  0, \text{~otherwise}
\end{cases}
$$
where $1 \leq m \leq k, 1 \leq n \leq p$. Form a binary matrix $V^{P}$ of
size $pk \times p$ by taking $v_j^{P}$, as columns for $j=1,\dots,p.$ 

It can be verified that the matrix $V^{P}$ satisfies the following properties.
\begin{enumerate}
\item $V^{P}$ has $k$ number of row-blocks with each row-block being
  of size $p$. Each column $v^{P}_{j}$ of $V^{P}$ has exactly $k$ number of ones
  and
  contains a single $1$-valued entry in each block. Also, due to the
  construction, it is easy to see that every row of
  $V^{P}$ contains a single
  $1$-valued entry. Therefore, each row-block is a column (or row)
  permutation of an identity matrix.
\item The density of $V^{P}$ is $\frac{1}{p}.$
\item For $i\neq j,$ there are no common points between any two
  distinct $k-$tuples $d^{P}_{i}$ and $d^{P}_{j}$. This is true because
  $P+f_{i}$ and $P+f_{j}$ have no common root.  As a result there is no
  overlap (i.e., no two columns contain $1$ at the same position)
  between any two distinct columns of $V^{P}$.
\end{enumerate}

We now discuss the construction procedure to produce a unitary matrix
from $V^{P}$. Let $U_{P}$ be a $k\times k$ unitary matrix. A new matrix
$\Phi^P$ is obtained by replacing, in each column of  $V_P$, every 1-valued
entry with a distinct row of $U_P$. The $0$-valued entries are
replaced by a row of zeros. It is clear that the size of the matrix
$\Phi^P$ is $pk \times pk$. The orthonormality of the rows of $\Phi^P$
follows from the fact that $U_P$ is unitary.

A new matrix $\Phi$ is constructed by concatenating $\Phi^P$
for $P\in S^{p}$. The size of $\Phi$ is $pk \times (pk\times p^{r})$.
Let $\alpha =\max_{P,i,j} |u_{P}(i,j)|,$ where $u_{k,P}(i,j)$ denotes the
$(i,j)-$th entry of $U_{P}$. The following theorem bounds the coherence of
$\Phi$
\begin{theorem}
  The coherence $\mu_{\Phi}$ of $\Phi$ is at most $\min(r\alpha^{2},1)$
\end{theorem}
\textit{Proof}: The proof follows from the definition of $\alpha$ and from
the fact that any two distinct polynomials $P^{1}$ and $P^{2}$ belonging to
$S^{p}$ can have at most $r$ number of common roots.

\begin{theorem}
  For $r<k\leq p,$ where $p$ is a prime or a prime power, if there is a $k\times k$
  unitary matrix such that the largest of the absolute values of its entries ($\alpha$)
  satisfies $r\alpha^{2}<1$, then there exists a CS matrix which is
  a union of $p^{r}$ number of orthonormal bases, with coherence being at most
  $r\alpha^{2}$ and with density $\frac{1}{p}$.  
\end{theorem}

\begin{remark}
  One can take DCT (Discrete Cosine Transform) and DFT (Discrete Fourier
  Transform) matrices in the real and complex cases respectively. The DCT matrix
  is defined for $0 \leq i \leq k-1, 0 \leq j \leq k-1$ as
  $$
  U(i,j) =
  \begin{cases}
    \sqrt{1/k}\cos((\pi/k)(j + 0.5)i) & \text{for~} i = 0 \\
    \sqrt{2/k}\cos((\pi/k)(j + 0.5)i) & \text{otherwise}
  \end{cases}
  $$
  For the DCT matrix $\alpha \leq \sqrt{2/k}$ and therefore the coherence
  is at most $2r/k$. In the complex case, it can be seen that
  the coherence is at most $r/k$, when the DFT matrix is used.

\end{remark}

\subsection{A special case}
In this section we discuss a special case with  $r=1, k=p.$ Observe that
for $r=1,$  any two distinct polynomials, $P^{(1)}$ and $P^{(2)}$ belonging
to $S^{p}$ have exactly one common root (i.e., $0$). Therefore,
the intersection between, $v^{P^{(1)}}_{j}$ ($j-$th column of  $V^{P^{(1)}}$)
and $v^{P^{(2)}}_{i}$ ($i-$th column of $V^{P^{(2)}}$) is exactly one.  

Let $H_{p\times p}$ be an orthogonal matrix whose entries are uni-modular
(i. e., $|h(i,j)|= 1$). In the real case, one can take $H$ as the Hadamard
matrix of order $p$. For $p=2^{i}: i\geq 2$ Hadamard matrices of order
$p$ are known to exist. In the complex case, $H$ can be chosen as the
discrete Fourier transform matrix (DFT). In the construction process,
we replace the unitary matrix $U_{P}$ with $\frac{1}{\sqrt{p}}H$. Then the following hold,
\begin{enumerate}
\item The inner-product between two columns of $\Phi$
  corresponding to the same polynomial is zero.
\item The absolute value of the inner product between two columns of
  $\Phi$ corresponding to two different polynomials is $\frac{1}{p}$.
\end{enumerate}
As a result $\Phi$ becomes a union of $p$ mutually unbiased bases with
coherence $\mu_{\Phi} = \frac{1}{p}$.
\section{Construction for the composite case}
\label{sec:conscomp}
For the composite case, we use the following composition
rule given in~\cite{pra_2016} for combining binary matrices. The following
result has been proved there.
\begin{lemma}[Lemma 4 in \cite{pra_2016}]
  \label{lem:comprule}
 For $i=1,2$, let $\Psi_{i}$ be a binary (containing $0,1$) matrix of size
  $m_{i} \times
  M_{i}$ consisting of $k_{i}$ number of row blocks each
  having a size $n_{i}$ so that the intersection between any two
  columns is at most $r_{i}$ and assume that
  $r = \max\{r_{1},r_{2}\} < k \leq \min\{k_{1},k_{2}\}
  \leq \min\{n_{1},n_{2}\}$. Then, the composition rule, denoted by $\ast,$
      produces a matrix $\Psi=\Psi_{1}\ast \Psi_{2}$ of size
      $n_{1}n_{2}k \times M_{1}M_{2}$ containing
      $k$ number of row blocks each having size $n_{1}n_{2}$
      with the intersection between any two columns being at most $r$ and density of $\Psi$ being $\frac{1}{n_{1}n_{2}}$.\quad
		
\end{lemma}
Let $p$ and $q$ be two distinct primes or prime powers.
With $r<k \leq \min\{p,q\},$ $P \in S^{p}$ and
$Q \in S^{q},$ we apply the composition rule on the
matrices $V^{P}$ and $V^{Q}$ to obtain a new binary matrix $V^{P,Q}=V^{P}\ast V^{Q}$.
It is easy to see that $V^{P,Q}$ satisfies the following properties,

\begin{enumerate}
\item The size of $V^{P,Q}$ is $pqk \times pq.$  
\item $V^{P,Q}$ has $k$ number of row-blocks and each block is of size $pq$
\item There is no overlap between any two distinct columns of $V^{P,Q}.$
\item The density of $V^{P,Q}$ is $\frac{1}{pq}.$
\end{enumerate}

Let $U$ be a $k \times k$ unitary matrix. For each column of $V^{P,Q}$ we replace each
of its $1$-valued entries with a distinct row of $U$ to obtain
a new unitary matrix $\Psi^{P,Q}$ of size $pqk \times pqk$.

The matrix $\Psi$ is constructed by concatenating
$\Psi^{P,Q}$ for $P\in S^{p}$ and
$Q\in S^{q}.$ Let $\alpha = \max_{i,j} |u_{i,j}|$ where
$u_{i,j}$ is the $(i,j)$th element in $U$. The following
properties of $\Psi$ can be easily established.
\begin{enumerate}
\item The size of  $\Psi$ is $pqk \times (pqk \times(pq)^{r}).$ 
\item $\Psi$ is a union of $(pq)^{r}$ number of orthonormal bases. 
\end{enumerate}

We show next that the coherence $\mu_{\Psi}$ of $\Psi$ is at most
$\min(r\alpha^{2},1)$. For the proof of this result, we first give
the concatenation property of the composition rule.
\begin{lemma}
  \label{lem:concat}
  Let $V\ast W$ be the result of composition of the matrices $V$ and $W$
  using the rule given in~\cite{pra_2016}.
  Then $[V_1,V_2]\ast [W_1,W_2] = [V_1\ast W_1,V_1\ast W_2,V_2\ast W_1,
 V_2\ast W_2]$ where $[V,W]$ denotes the column-wise
  concatenation of the two matrices $V$ and $W$.
\end{lemma}
\begin{proof}
  The composition rule given in~\cite{pra_2016} is a column-wise operation.
  For constructing $V\ast W$ the support of each column of
  $V$ is combined in an appropriate manner with the support of each column of
  $W$. Therefore, the procedure maintains the concatenation property.
\end{proof}

\begin{theorem}
  The coherence $\mu_{\Psi}$ of $\Psi$ is at most
  $\min(r\alpha^{2},1)$.
\end{theorem}
\begin{proof}
Let $P^{(1)},P^{(2)}\in S^{p}$ and  $Q^{(1)},Q^{(2)}\in S^{q},$ such that
$P^{(1)} \neq P^{(2)}$ or $Q^{(1)} \neq Q^{(2)}$. Consider the
composition of the column concatenated matrices, $[V^{P^{(1)}}, V^{P^{(1)}}]$
$[V^{Q^{(1)}},V^{Q^{(2)}}]$. Note that $V^{P^{(1)}}$ and $V^{P^{(2)}}$ have at most
$r$-intersections among any pairs of their columns. Similarly $V^{Q^{(1)}}$ and
$V^{Q^{(2)}}$ also have at most $r$-intersections among their columns.
Therefore, from Lemma \ref{lem:comprule}, the resultant matrix after
composition $[V^{P^{(1)}}, V^{P^{(2)}}] * [V^{Q^{(1)}},V^{Q^{(2)}}]$,
has at most $r$-intersections among its columns. However, from
Lemma~\ref{lem:concat}, we have $[V^{P^{(1)}}, V^{P^{(1)}}] \ast [V^{Q^{(1)}},V^{Q^{(2)}}] =$
$[V^{P^{(1)}}\ast V^{Q^{(1)}},V^{P^{(1)}}\ast V^{Q^{(2)}},V^{P^{(2)}} \ast V^{Q^{(1)}},V^{P^{(2)}} \ast V^{Q^{(2)}}]$. This proves that
$V^{P^{(1)}}\ast V^{Q^{(1)}}$ and $V^{P^{(2)}}\ast V^{Q^{(2)}}$ have at most
$r$-intersections among any pair of their columns.
\end{proof}
Using the above result recursively, we have the following result for
general $m$.
\begin{theorem}
  Let $m=p_{1}\dots p_{t},$ where $p_{1},\dots , p_{t}$ are primes or prime powers
  and $r<k\leq \min \{p_{1},\dots ,p_{t}\}$ and $U$ be a $k\times k$ unitary matrix with largest absolute entry $\alpha.$ Then, there
  exists a CS matrix of size $mk\times (mk\times m^{r})$, which is a
  union of $m^{r}$ number of orthonormal bases, with coherence being at
  most $\min(r\alpha^{2},1)$ and density being $\frac{1}{m}.$  
\end{theorem}
Proof: For $P_{1}\in S^{p_{1}},$ using the construction procedure described in section IV,  a binary matrix $V^{P_{1}}$ is obtained. Now, applying the composition
rule in \cite{pra_2016} on binary
matrices $V^{P_{i}}$ successively for $i=1, \dots, t$, we obtain a new binary matrix
$V^{P_{1},\dots,P_{t}}=V^{P_{1}}\ast V^{P_{2}} \dots \ast V^{P_{t}}$ of size $mk \times m.$ By construction, every column of $V^{P_{1},\dots,P_{t}}$ contains $k$ number of $1$-valued
entries and there is no intersection between any two columns of
$V^{P_{1},\dots,P_{t}}.$
Let $P^{(1)}_{i}, P^{(2)}_{i}\in S^{p_{i}}$ such that $P^{(1)}_{i} \neq P^{(2)}_{i}$
for at least one $i\in \{1,\dots, t\},$ then by composition rule the intersection
between any column of $V^{P^{(1)}_{1},\dots,P^{(1)}_{t}}$ and any column of
$V^{P^{(2)}_{1},\dots,P^{(2)}_{t}}$ is at most $r.$ This can be seen by iteratively
applying the concatenation Lemma~\ref{lem:concat} to expand
$[V^{P_1^{(1)}},V^{P_1^{(2)}}] \ast [V^{P_2^{(1)}},V^{P_2^{(2)}}] \ast \cdots
\ast [V^{P_t^{(1)}},V^{P_t^{(2)}}]$ into individual composited matrices.
Now, as described previously, embedding a unitary matrix $U_{k\times k}$ into $V^{P_{1},\dots,P_{t}},$ we can construct a unitary matrix
$\Psi^{P_{1},\dots,P_{t}}_{U}$  of size $mk \times mk.$
The matrix $\Psi^{r}_{m,k}$ is constructed by
concatenating $\Psi^{P_{1},\dots,P_{t}}_{U}$ by taking
$P_{i}\in S^{p_{i}}.$ 
Therefore,
$\Psi^{r}_{m,k}$ is a union of $m^{r}$ number of orthonormal bases, with
coherence being at most $\min(r\alpha^{2},1)$ and density being $\frac{1}{m}$.

%

\section{Numerical Simulations}
This section presents the numerical results for demonstrating the
recovery performance of frames constructed via embedding DCT matrix.
The column size of the constructed matrix is $mk*m^{r}$ and the coherence is
at most $\frac{2r}{k}$. For obtaining small coherence, it is necessary to
consider $r \ll k$. Since $k$ is the smallest prime factor in $m$,
for large values of $k$, $m$ is also proportionately large. For example,
if $m = k = 17$ and $r=1$, the column size is in the order of $10^3$, whereas
if $r=2$, the column size is in the order of $10^4$. 
In the results shown here, as an example, it is assumed that $r=1$ and
$m=k\leq 17$ to ease the computational demand.
The comparison is performed with respect to Gaussian random matrices. A total
of $1000$ different signals are considered for each sparsity level and
the reconstruction performance is measured. The reconstruction is
considered good if the SNR (defined below) is greater than $100$dB.
If $x$ is the original signal and $\hat{x}$ is the estimated signal, then
\[SNR = 10\log_{10}\frac{\|x\|}{\|\hat{x}-x\|}.\]

The solutions are computed using the orthogonal matching
pursuit (OMP) algorithm. The stopping criterion is considered to be
the actual sparsity of the signal. Fig. 1 provides comparison
of the success rates of reconstructions between  structured frames and
their Gaussian counterparts. For a given sparsity level, if $90$ percent
of the signals are reconstructed accurately (i.e., their SNR values are above
the
threshold of 100 dB) then we consider that the performance is good for that
sparsity level.  In the above figure, only the performance for sparsity
levels satisfying the aforementioned condition is shown.
It can be seen from this plot that the constructed structured frames show
superior performance compared to Gaussian random matrices.
\begin{figure}
  \label{fig1}   
\centering
\includegraphics[scale=0.19]{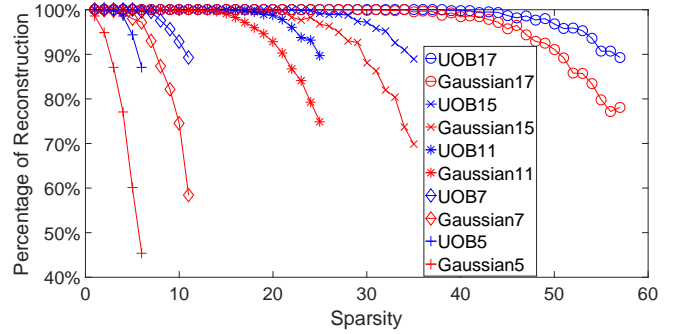}
\caption{Comparison of reconstruction performance of unions of
  orthonormal bases for different $m$ and Gaussian random matrices of the
  same size. UOB stands for Union of Orthonormal Bases. 
  The sizes of the matrices are given by $mk \times (mk*m^r)$
  where $k=m$ and $r=1$. The results are reported for $m=5,7,11,15,17$.
  UOB17, for e.g., represents the results for the matrix with $m=k=17$ and
  $r=1$. The $x$
  and $y$ axes in this plot respectively represent the sparsity (that is, zero
  norm of solution to be recovered) of solution and success rate.
  The performance is only shown for sparsity levels for which at least
  $90\%$ of the signals have been accurately reconstructed.}
\end{figure}

\section{Conclusion}
\label{sec:conclusions}
In the present work, we have constructed union of orthonormal bases for general sizes.  The matrices for
sizes governed by primes and their powers are constructed using polynomials over finite fields. For
constructing frames of general sizes, a recently proposed composition rule has been used. Construction of
mutually unbiased bases has also been given for the prime power cases. Numerical
results show that the constructed structured
frames show superior performance when compared to Gaussian Random matrices
of the same sizes.

\end{document}